\documentclass{amsart}
\usepackage{mathptmx}
\usepackage{amscd,amssymb,enumerate,bm}

\newtheorem{theorem}{Theorem}[section]
\newtheorem{lemma}[theorem]{Lemma}
\newtheorem{corollary}[theorem]{Corollary}
\newtheorem{proposition}[theorem]{Proposition}

\theoremstyle{definition}
\newtheorem{definition}[theorem]{Definition}
\newtheorem{example}[theorem]{Example}
\newtheorem{remark}[theorem]{Remark}
\newtheorem{conjecture}[theorem]{Conjecture}
\numberwithin{equation}{theorem}

\def\cm{\divideontimes}
\def\bar{\overline}
\def\phi{\varphi}
\def\to{\longrightarrow}
\def\mapsto{\longmapsto}
\renewcommand{\mod}{\,\operatorname{mod}\,}

\def\End{\operatorname{End}}

\def\Hom{\operatorname{Hom}}

\def\reg{\operatorname{reg}}

\def\hat{\widehat}

\def\ge{\geqslant}
\def\le{\leqslant}

\def\bsz{{\boldsymbol{z}}}
\def\bslambda{{\boldsymbol{\lambda}}}

\def\fraka{{\mathfrak a}}

\def\frakm{{\mathfrak m}}

\def\frakp{{\mathfrak p}}

\def\FF{{\mathbb F}}
\def\NN{{\mathbb N}}
\def\QQ{{\mathbb Q}}

\def\ZZ{{\mathbb Z}}

\def\calC{{\mathcal C}}
\def\calF{{\mathcal F}}
\def\calR{{\mathcal R}}
\def\calT{{\mathcal T}}

\begin{document}
\title{Rings of Frobenius operators}

\author{Mordechai Katzman}
\address{Department of Pure Mathematics, University of Sheffield, Hicks Building, Sheffield,~S3~7RH,~UK}
\email{M.Katzman@sheffield.ac.uk}

\author{Karl Schwede}
\address{Department of Mathematics, The Pennsylvania State University, University Park, PA~16802,~USA}
\email{schwede@math.psu.edu}

\author{Anurag K. Singh}
\address{Department of Mathematics, University of Utah, 155 S 1400 E, Salt Lake City, UT~84112,~USA}
\email{singh@math.utah.edu}

\author{Wenliang Zhang}
\address{Department of Mathematics, University of Nebraska, Lincoln, NE~68505,~USA}
\email{wzhang15@unl.edu}

\thanks{M.K.~was supported by EPSRC grant EP/I031405/1, K.S.~by NSF grant DMS~1064485 and a Sloan Fellowship, A.K.S.~by NSF grant DMS~1162585, and W.Z.~by NSF grant DMS~1247354.}

\date{\today}

\begin{abstract}
Let $R$ be a local ring of prime characteristic. We study the ring of Frobenius operators $\calF(E)$, where $E$ is the injective hull of the residue field of $R$. In particular, we examine the finite generation of $\calF(E)$ over its degree zero component $\calF^0(E)$, and show that $\calF(E)$ need not be finitely generated when $R$ is a determinantal ring; nonetheless, we obtain concrete descriptions of $\calF(E)$ in good generality that we use, for example, to prove the discreteness of $F$-jumping numbers for arbitrary ideals in determinantal rings.
\end{abstract}
\maketitle

\section{Introduction}

Lyubeznik and Smith~\cite{LyubeznikSmith} initiated the systematic study of rings of Frobenius operators and their applications to tight closure theory. Our focus here is on the Frobenius operators on the injective hull of $R/\frakm$, when $(R,\frakm)$ is a complete local ring of prime characteristic.

\begin{definition}
Let $R$ be a ring of prime characteristic $p$, with Frobenius endomorphism~$F$. Following \cite[Section~3]{LyubeznikSmith}, we set $R\{F^e\}$ to be the ring extension of $R$ obtained by adjoining a noncommutative variable $\chi$ subject to the relations $\chi r=r^{p^e}\chi$ for all $r\in R$.

Let $M$ be an $R$-module. Extending the $R$-module structure on $M$ to an $R\{F^e\}$-module structure is equivalent to specifying an additive map $\phi\colon M\to M$ that satisfies
\[
\phi(rm)\ =\ r^{p^e}\phi(m)\qquad\text{ for each }r\in R\text{ and }m\in M\,.
\]
Define $\calF^e(M)$ to be the set of $R\{F^e\}$-module structures on $M$; this is an Abelian group with a left $R$-module structure, where $r\in R$ acts on $\phi\in\calF^e(M)$ to give the composition~$r\circ\phi$. Given elements $\phi\in\calF^e(M)$ and $\phi'\in\calF^{e'}(M)$, the compositions $\phi\circ\phi'$ and $\phi'\circ\phi$ are elements of the module $\calF^{e+e'}(M)$. Thus,
\[
\calF(M)\ =\ \calF^0(M)\oplus\calF^1(M)\oplus\calF^2(M)\oplus\cdots
\]
has a ring structure; this is the \emph{ring of Frobenius operators} on $M$.
\end{definition}

Note that $\calF(M)$ is an $\NN$-graded ring; it is typically not commutative. The degree~$0$ component $\calF^0(M)=\End_R(M)$ is a subring, with a natural $R$-algebra structure. Lyubeznik and Smith~\cite[Section~3]{LyubeznikSmith} ask whether~$\calF(M)$ is a finitely generated ring extension of $\calF^0(M)$. From the point of view of tight closure theory, the main cases of interest are where $(R,\frakm)$ is a complete local ring, and the module~$M$ is the local cohomology module $H^{\dim R}_\frakm(R)$ or the injective hull of the residue field, $E_R(R/\frakm)$, abbreviated~$E$ in the following discussion. In the former case, the algebra $\calF(M)$ is finitely generated under mild hypotheses, see Example~\ref{example:intro}.\ref{example:lc}; an investigation of the latter case is our main focus here.

It follows from Example~\ref{example:intro}.\ref{example:lc} that for a Gorenstein complete local ring $(R,\frakm)$, the ring~$\calF(E)$ is a finitely generated extension of $\calF^0(E)\cong R$. This need not be true when~$R$ is not Gorenstein: Katzman~\cite{Katzman} constructed the first such examples. In Section~\ref{sec:ring:structure} we study the finite generation of $\calF(E)$, and provide descriptions of $\calF(E)$ even when it is not finitely generated: this is in terms of graded subgroup of the anticanonical cover of $R$, with a Frobenius-twisted multiplication structure, see Theorem~\ref{theorem:main}.

Section~\ref{sec:q:gorenstein} studies the case of $\QQ$-Gorenstein rings. We show that $\calF(E)$ is finitely generated (though not necessarily principally generated) if $R$ is $\QQ$-Gorenstein with index relatively prime to the characteristic, Proposition~\ref{prop:q:gor}; the dual statement for the Cartier algebra was previously obtained by Schwede in~\cite[Remark~4.5]{Schwede}. We also construct a~$\QQ$-Gorenstein ring for which the ring $\calF(E)$ is \emph{not} finitely generated over $\calF^0(E)$; in fact, we conjecture that this is always the case for a $\QQ$-Gorenstein ring whose index is a multiple of the characteristic, see Conjecture~\ref{conj:Q:gor}.

In Section~\ref{sec:determinantal} we show that $\calF(E)$ need not be finitely generated for determinantal rings, specifically for the ring~$\FF[X]/I$, where $X$ is a $2\times 3$ matrix of variables, and $I$ is the ideal generated by its size $2$ minors; this proves a conjecture of Katzman,~\cite[Conjecture~3.1]{Katzman}. The relevant calculations also extend a result of Fedder,~\cite[Proposition~4.7]{Fedder}.

One of the applications of our study of $\calF(E)$ is the discreteness of $F$-jumping numbers; in Section~\ref{sec:gauge} we use the description of $\calF(E)$, combined with the notion of gauge boundedness, due to Blickle~\cite{Blickle:JAG}, to obtain positive results on the discreteness of $F$-jumping numbers for new classes of rings including determinantal rings, see Theorem~\ref{theorem:gauge}. In the last section, we obtain results on the linear growth of Castelnuovo-Mumford regularity for rings with finite Frobenius representation type; this is also with an eye towards the discreteness of $F$-jumping numbers.

To set the stage, we summarize some previous results on the rings $\calF(M)$.

\begin{example}
\label{example:intro}
Let $R$ be a ring of prime characteristic.
\begin{enumerate}[\rm(1)]
\item For each $e\ge0$, the left $R$-module $\calF^e(R)$ is free of rank one, spanned by $F^e$; this is~\cite[Example~3.6]{LyubeznikSmith}. Hence, $\calF(R)\cong R\{F\}$.

\smallskip
\item\label{example:lc}
Let $(R,\frakm)$ be a local ring of dimension $d$. The Frobenius endomorphism~$F$ of $R$ induces, by functoriality, an additive map
\[
F\colon H^d_\frakm(R)\to H^d_\frakm(R)\,,
\]
which is the natural \emph{Frobenius action} on $H^d_\frakm(R)$. If the ring $R$ is complete and $S_2$, then~$\calF^e(H^d_\frakm(R))$ is a free left $R$-module of rank one, spanned by $F^e$; for a proof of this, see~\cite[Example~3.7]{LyubeznikSmith}. It follows that
\[
\calF\big(H^d_\frakm(R)\big)\ \cong\ R\{F\}\,.
\]
In particular, $\calF(H^d_\frakm(R))$ is a finitely generated ring extension of $\calF^0(H^d_\frakm(R))$.

\smallskip
\item Consider the local ring $R=\FF[[x,y,z]]/(xy,yz)$ where $\FF$ is a field, and set $E$ to be the injective hull of the residue field of $R$. Katzman~\cite{Katzman} proved that $\calF(E)$ is not a finitely generated ring extension of $\calF^0(E)$.

\smallskip
\item Let $(R,\frakm)$ be the completion of a Stanley-Reisner ring at its homogeneous maximal ideal, and let $E$ be the injective hull of $R/\frakm$. In
\cite{ABZ} \`Alvarez, Boix, and Zarzuela obtain necessary and sufficient conditions for the finite generation of $\calF(E)$. Their work yields, in particular, Cohen-Macaulay examples where $\calF(E)$ is not finitely generated over $\calF^0(E)$. By \cite[Theorem~3.5]{ABZ}, $\calF(E)$ is either $1$-generated or infinitely generated as a ring extension of $\calF^0(E)$ in the Stanley-Reisner case.
\end{enumerate}
\end{example}

\begin{remark}
\label{remark:adjoint}
Let $R^{(e)}$ denote the $R$-bimodule that agrees with $R$ as a left $R$-module, and where the right module structure is given by
\[
x \cdot r\ =\ r^{p^e}x\qquad\text{ for all }r\in R\text{ and }x\in R^{(e)}\,.
\]
For each $R$-module $M$, one then has a natural isomorphism
\[
\calF^e(M)\ \cong\ \Hom_R\big(R^{(e)}\otimes_RM,\ M\big)
\]
where $\phi\in\calF^e(M)$ corresponds to $x\otimes m\mapsto x\phi(m)$ and $\psi\in\Hom_R(R^{(e)}\otimes_RM,\ M)$ corresponds to $m\mapsto\psi(1\otimes m)$; see~\cite[Remark~3.2]{LyubeznikSmith}.
\end{remark}

\begin{remark}
Let $R$ be a Noetherian ring of prime characteristic. If $M$ is a Noetherian $R$-module, or if $R$ is complete local and $M$ is an Artinian $R$-module, then each graded component~$\calF^e(M)$ of~$\calF(M)$ is a finitely generated left $R$-module, and hence also a finitely generated left $\calF^0(M)$-module; this is~\cite[Proposition~3.3]{LyubeznikSmith}.
\end{remark}

\begin{remark}
\label{remark:blickle}
Let $R$ be a complete local ring of prime characteristic $p$; set $E$ to be the injective hull of the residue field of $R$. Let $A$ be a complete regular local ring with $R=A/I$. By \cite[Proposition~3.36]{Blickle:thesis}, one then has an isomorphism of $R$-modules
\[
\calF^e(E)\ \cong\ \frac{I^{[p^e]}:_AI}{I^{[p^e]}}\,.
\]
\end{remark}

\section{Twisted multiplication}
\label{sec:twisted}

Let $R$ be a complete local ring of prime characteristic; let $E$ denote the injective hull of the residue field of $R$. In Theorem~\ref{theorem:main} we prove that $\calF(E)$ is isomorphic to a subgroup of the anticanonical cover of $R$, with a twisted multiplication structure; in this section, we describe this twisted construction in broad generality:

\begin{definition}
\label{defn:cm}
Given an $\NN$-graded commutative ring $\calR$ of prime characteristic $p$, we define a new ring~$\calT(\calR)$ as follows: Consider the Abelian group
\[
\calT(\calR)\ =\ \bigoplus_{e\ge0}\calR_{p^e-1}\,,
\]
and define a multiplication $\cm$ on $\calT(\calR)$ by
\[
a\cm b\ =\ ab^{p^e}\qquad\text{ for }a\in{\calT(\calR)}_e\text{ and }b\in{\calT(\calR)}_{e'}\,.
\]
\end{definition}

It is a straightforward verification that $\cm$ is an associative binary operation; the prime characteristic assumption is used in verifying that $+$ and $\cm$ are distributive. Moreover, for elements $a\in{\calT(\calR)}_e$ and $b\in{\calT(\calR)}_{e'}$ one has
\[
ab^{p^e}\ \in\ \calR_{p^e-1+p^e(p^{e'}-1)}=\calR_{p^{e+{e'}}-1}\,,
\]
and hence
\[
\calT(\calR)_e\cm\calT(\calR)_{e'}\ \subseteq\ \calT(\calR)_{e+e'}\,.
\]
Thus, $\calT(\calR)$ is an $\NN$-graded ring; we abbreviate its degree $e$ component ${\calT(\calR)}_e$ as $\calT_e$. The ring $\calT(\calR)$ is typically not commutative, and need not be a finitely generated extension ring of $\calT_0$ even when~$\calR$ is Noetherian:

\begin{example}
\label{example:cm}
We examine $\calT(\calR)$ when $\calR$ is a standard graded polynomial ring over a field $\FF$. We show that $\calT(\calR)$ is a finitely generated ring extension of~$\calT_0=\FF$ if $\dim\calR\le 2$, and that $\calT(\calR)$ is not finitely generated if $\dim\calR\ge 3$.

\begin{enumerate}[\rm(1)]
\item If $\calR$ is a polynomial ring of dimension $1$, then $\calT(\calR)$ is commutative and finitely generated over $\FF$: take $\calR=\FF[x]$, in which case $\calT_e=\FF\cdot x^{p^e-1}$ and
\[
x^{p^e-1}\cm x^{p^{e'}-1}\ =\ x^{p^{e+e'}-1}\ =\ x^{p^{e'}-1}\cm x^{p^e-1}\,.
\]
Thus, $\calT(\calR)$ is a polynomial ring in one variable.

\smallskip
\item When $\calR$ is a polynomial ring of dimension $2$, we verify that $\calT(\calR)$ is a noncommutative finitely generated ring extension of $\FF$. Let $\calR=\FF[x,y]$. Then
\[
x^{p-1}\cm y^{p-1}\ =\ x^{p-1}y^{p^2-p} \qquad\text{ whereas }\qquad
y^{p-1}\cm x^{p-1}\ =\ x^{p^2-p}y^{p-1}\,,
\]
so $\calT(\calR)$ is not commutative. For finite generation, it suffices to show that
\[
\calT_{e+1}\ =\ \calT_1\cm\calT_e\qquad\text{ for each }e\ge1\,.
\]
Set $q=p^e$ and consider the elements
\[
x^iy^{p-1-i}\in\calT_1\,,\quad{0\le i\le p-1}\quad\text{ and }\quad
x^jy^{q-1-j}\in\calT_e\,,\quad{0\le j\le q-1}\,.
\]
Then $\calT_1\cm\calT_e$ contains the elements
\[
\big(x^iy^{p-1-i}\big)\cm\big(x^jy^{q-1-j}\big)\ =\ x^{i+pj}y^{pq-pj-i-1}\,,
\]
for $0\le i\le p-1$ and $0\le j\le q-1$, and these are readily seen to span $\calT_{e+1}$. Hence, the degree $p-1$ monomials in $x$ and $y$ generate $\calT(\calR)$ as a ring extension of $\FF$.

\smallskip
\item\label{example:dim3}
For a polynomial ring $\calR$ of dimension $3$ or higher, the ring $\calT(\calR)$ is noncommutative and not finitely generated over $\FF$. The noncommutativity is immediate from (2); we give an argument that $\calT(\calR)$ is not finitely generated for $\calR=\FF[x,y,z]$, and this carries over to polynomial rings $\calR$ of higher dimension.

Set $q=p^e$ where $e\ge 2$. We claim that the element
\[
xy^{q/p-1}z^{q-q/p-1}\ \in\ \calT_e
\]
does not belong to $\calT_{e_1}\cm\calT_{e_2}$ for integers $e_i<e$ with $e_1+e_2=e$. Indeed, $\calT_{e_1}\cm\calT_{e_2}$ is spanned by the monomials
\[
\big(x^iy^jz^{q_1-i-j-1}\big)\cm\big(x^ky^lz^{q_2-k-l-1}\big)\ =\
x^{i+q_1k}y^{j+q_1l}z^{q-i-j-q_1k-q_1l-1}
\]
where $q_i=p^{e_i}$ and
\begin{align*}
0\le i\le q_1-1\,,\qquad &\qquad 0\le j\le q_1-1-i\,,\\
0\le k\le q_2-1\,,\qquad &\qquad 0\le l\le q_2-1-k\,,
\end{align*}
so it suffices to verify that the equations
\[
i+q_1k=1\qquad\text{ and }\qquad j+q_1l=q/p-1
\]
have no solution for integers $i,j,k,l$ in the intervals displayed above. The first of the equations gives $i=1$, which then implies that $0\le j\le q_1-2$. Since $q_1$ divides $q/p$, the second equation gives $j\equiv-1\mod q_1$. But this has no solution with $0\le j\le q_1-2$.
\end{enumerate}
\end{example}

\section{The ring structure of $\calF(E)$}
\label{sec:ring:structure}

We describe the ring of Frobenius operators $\calF(E)$ in terms of the symbolic Rees algebra~$\calR$ and the twisted multiplication structure $\calT(\calR)$ of the previous section. First, a notational point: $\omega^{[p^e]}$ below denotes the iterated Frobenius power of an ideal $\omega$, and $\omega^{(n)}$ its symbolic power, which coincides with reflexive power for divisorial ideals $\omega$. We realize that the notation $\omega^{[n]}$ is sometimes used for the reflexive power, hence this note of caution. We start with the following observation:

\begin{lemma}
\label{lemma:symb:frob}
Let $(R,\frakm)$ be a normal local ring of characteristic $p>0$. Let $\omega$ be a divisorial ideal of $R$, i.e., an ideal of pure height one. Then for each integer $e\ge1$, the map
\[
H^{\dim R}_\frakm\big(\omega^{[p^e]}\big)\to H^{\dim R}_\frakm\big(\omega^{(p^e)}\big)
\]
induced by the inclusion $\omega^{[p^e]}\subseteq\omega^{(p^e)}$, is an isomorphism.
\end{lemma}

\begin{proof}
Set $d=\dim R$. Since $R$ is normal and $\omega$ has pure height one, $\omega R_\frakp$ is principal for each prime ideal $\frakp$ of height one; hence $\big(\omega^{(p^e)}/\omega^{[p^e]}\big)R_\frakp=0$. It follows that
\[
\dim\big(\omega^{(p^e)}/\omega^{[p^e]}\big)\ \le\ d-2\,,
\]
which gives the vanishing of the outer terms of the exact sequence
\[
\CD
H^{d-1}_\frakm(\omega^{(p^e)}/\omega^{[p^e]})@>>> H^d_\frakm(\omega^{[p^e]})
@>>> H^d_\frakm(\omega^{(p^e)})@>>> H^d_\frakm(\omega^{(p^e)}/\omega^{[p^e]})\,,
\endCD
\]
and thus the desired isomorphism.
\end{proof}

\begin{definition}
Let $R$ be a normal ring that is either complete local, or $\NN$-graded and finitely generated over $R_0$. Let $\omega$ denote the canonical module of $R$. The symbolic Rees algebra
\[
\calR\ =\ \bigoplus_{n\ge0}\omega^{(-n)}
\]
is the \emph{anticanonical cover} of $R$; it has a natural $\NN$-grading where $\calR_n=\omega^{(-n)}$.
\end{definition}

\begin{theorem}
\label{theorem:main}
Let $(R,\frakm)$ be a normal complete local ring of characteristic $p>0$. Set $d$ to be the dimension of $R$. Let $\omega$ denote the canonical module of $R$, and identify $E$, the injective hull of the $R/\frakm$, with $H^d_\frakm(\omega)$.

\begin{enumerate}[\ \rm(1)]
\item Then $\calF(E)$, the ring of Frobenius operators on $E$, may be identified with
\[
\bigoplus_{e\ge0}\omega^{(1-p^e)}\,F^e\,,
\]
where $F^e$ denotes the map $H^d_\frakm(\omega)\to H^d_\frakm(\omega^{(p^e)})$ induced by $\omega\to\omega^{[p^e]}$.

\item Let $\calR$ be the anticanonical cover of $R$. Then one has an isomorphism of graded rings
\[
\calF(E)\ \cong\ \calT(\calR)\,,
\]
where $\calT(\calR)$ is as in Definition~\ref{defn:cm}.
\end{enumerate}
\end{theorem}

\begin{proof}
By Remark~\ref{remark:adjoint}, we have
\[
\calF^e\big(H^d_\frakm(\omega)\big)\ \cong\ \Hom_R\big(R^{(e)}\otimes_R H^d_\frakm(\omega),\ H^d_\frakm(\omega)\big)\,.
\]
Moreover,
\[
R^{(e)}\otimes_R H^d_\frakm(\omega)\ \cong\ H^d_\frakm(\omega^{[p^e]})\ \cong\ H^d_\frakm(\omega^{(p^e)})\,,
\]
where the first isomorphism of by \cite[Exercise~9.7]{24hours}, and the second by Lemma~\ref{lemma:symb:frob}. By similar arguments
\begin{eqnarray*}
\Hom_R\big(H^d_\frakm(\omega^{(p^e)}),\ H^d_\frakm(\omega)\big) & \cong & \Hom_R\big(H^d_\frakm(\omega\otimes_R\omega^{(p^e-1)}),\ H^d_\frakm(\omega)\big) \\
& \cong & \Hom_R\big(\omega^{(p^e-1)}\otimes_R H^d_\frakm(\omega),\ H^d_\frakm(\omega)\big) \\
&\cong & \Hom_R\big(\omega^{(p^e-1)},\ \Hom_R(H^d_\frakm(\omega),\ H^d_\frakm(\omega)\big)\big)\,,
\end{eqnarray*}
with the last isomorphism using the adjointness of $\Hom$ and tensor. Since $R$ is complete, the module above is isomorphic to
\[
\Hom_R\big(\omega^{(p^e-1)},\ R\big)\ \cong\ \omega^{(1-p^e)}\,.
\]

Suppose $\phi\in\calF^e(M)$ and $\phi'\in\calF^{e'}(M)$ correspond respectively to $aF^e$ and $a'F^{e'}$, for elements $a\in\omega^{(1-p^e)}$ and $a'\in\omega^{(1-p^{e'})}$. Then $\phi\circ\phi'$ corresponds to $aF^e\circ bF^{e'}=ab^{p^e}F^{e+e'}$, which agrees with the ring structure of $\calT(\calR)$ since $a\cm b=ab^{p^e}$.
\end{proof}

\begin{remark}
\label{remark:fedder:ring}
Let $R$ be a normal complete local ring of prime characteristic $p$; let $A$ be a complete regular local ring with $R=A/I$. Using Remark~\ref{remark:blickle} and Theorem~\ref{theorem:main}, its is now a straightforward verification that $\calF(E)$ is isomorphic, as a graded ring, to
\[
\bigoplus_{e\ge 0}\frac{I^{[p^e]}:_AI}{I^{[p^e]}}\,,
\]
where the multiplication on this latter ring is the twisted multiplication $\cm$. An example of the isomorphism is worked out in Proposition~\ref{prop:fedder}.
\end{remark}

\section{$\QQ$-Gorenstein rings}
\label{sec:q:gorenstein}

We analyze the finite generation of $\calF(E)$ when $R$ is $\QQ$-Gorenstein. The following result follows from the corresponding statement for Cartier algebras, \cite[Remark~4.5]{Schwede}, but we include it here for the sake of completeness:

\begin{proposition}
\label{prop:q:gor}
Let $(R,\frakm)$ be a normal $\QQ$-Gorenstein local ring of prime characteristic. Let $\omega$ denote the canonical module of $R$. If the order of $\omega$ is relatively prime to the characteristic of $R$, then
$\calF(E)$ is a finitely generated ring extension of $\calF^0(E)$.
\end{proposition}

\begin{proof}
Since $\calF^0(E)$ is isomorphic to the $\frakm$-adic completion of $R$, the proposition reduces to the case where the ring $R$ is assumed to be complete.

Let $m$ be the order of $\omega$, and $p$ the characteristic of $R$. Then $p\mod m$ is an element of the group $(\ZZ/m\ZZ)^\times$, and hence there exists an integer $e_0$ with $p^{e_0}\equiv1\mod m$. We claim that $\calF(E)$ is generated over $\calF^0(E)$ by ${[\calF(E)]}_{\le e_0}$.

We use the identification $\calF(E)=\calT(\calR)$ from Theorem~\ref{theorem:main}. Since $\omega^{(m)}$ is a cyclic module, one has
\[
\omega^{(n+km)}\ =\ \omega^{(n)}\omega^{(km)}\qquad\text{ for all integers }k,n\,.
\]
Thus, for each $e>e_0$, one has
\begin{align*}
\calT_{e-e_0}\cm\calT_{e_0}
&\ =\ \omega^{(1-p^{e-e_0})}\cm\omega^{(1-p^{e_0})}\\
&\ =\ \omega^{(1-p^{e-e_0})}\cdot\big(\omega^{(1-p^{e_0})}\big)^{[p^{e-e_0}]}\\
&\ =\ \omega^{(1-p^{e-e_0})}\cdot\omega^{(p^{e-e_0}(1-p^{e_0}))}\\
&\ =\ \omega^{(1-p^{e-e_0}+p^{e-e_0}-p^e)}\\
&\ =\ \omega^{(1-p^e)}\\
&\ =\ \calT_e\,,
\end{align*}
which proves the claim.
\end{proof}

We conjecture that Proposition~\ref{prop:q:gor} has a converse in the following sense:

\begin{conjecture}
\label{conj:Q:gor}
Let $(R,\frakm)$ be a normal $\QQ$-Gorenstein ring of prime characteristic, such that the order of the canonical module in the divisor class group is a multiple of the characteristic of $R$. Then~$\calF(E)$ is not a finitely generated ring extension of $\calF^0(E)$.
\end{conjecture}

\subsection*{Veronese subrings}

Let $\FF$ be a field of characteristic $p>0$, and $A=\FF[x_1,\dots,x_d]$ a polynomial ring. Given a positive integer $n$, we denote the $n$-th Veronese subring of $A$ by
\[
A_{(n)}\ =\ \bigoplus_{k\ge0}A_{nk}\,;
\]
this differs from the standard notation, e.g., \cite{GW}, since we reserve superscripts $(\phantom{m})^{(n)}$ for symbolic powers. The cyclic module $x_1\cdots x_dA$ is the graded canonical module for the polynomial ring~$A$. By \cite[Corollary~3.1.3]{GW}, the Veronese submodule
\[
\big(x_1\cdots x_dA\big)_{(n)}\ =\ \bigoplus_{k\ge0}\big[x_1\cdots x_dA\big]_{nk}
\]
is the graded canonical module for subring~$A_{(n)}$. Let $\frakm$ denote the homogeneous maximal ideal of $A_{(n)}$. The injective hull of $A_{(n)}/\frakm$ in the category of graded $A_{(n)}$-modules is
\begin{align*}
H^d_\frakm\Big(\big(x_1\cdots x_dA\big)_{(n)}\Big)\
& =\ \Big[H^d_\frakm(\big(x_1\cdots x_dA\big)\Big]_{(n)}\\
& =\ \left[\frac{A_{x_1\cdots x_d}}{\sum_i x_1\cdots x_dA_{x_1\cdots \hat{x}_i\cdots x_d}}\right]_{(n)}\,,
\end{align*}
see \cite[Theorem~3.1.1]{GW}. By \cite[Theorem~1.2.5]{GW}, this is also the injective hull in the category of all $A_{(n)}$-modules.

Let $R$ be the $\frakm$-adic completion of $A_{(n)}$. As it is $\frakm$-torsion, the module displayed above is also an $R$-module; it is the injective hull of $R/\frakm R$ in the category of $R$-modules.

\begin{proposition}
\label{prop:veronese}
Let $\FF$ be a field of characteristic $p>0$, and let $A=\FF[x_1,\dots,x_d]$ be a polynomial ring of dimension $d$. Let $n$ be a positive integer, and $R$ be the completion of the $n$-th Veronese subring of $A$ at its homogeneous maximal ideal. Set $E=M/N$ where
\[
M\ =\ R_{x_1^n\cdots x_d^n}
\]
and $N$ is the $R$-submodule spanned by elements $x_1^{i_1}\cdots x_d^{i_d}\in M$ with $i_k\ge 1$ for some~$k$; the module $E$ is the injective hull of the residue field of $R$.

Then $\calF^e(E)$ is the left $R$-module generated by the elements
\[
\frac{1}{x_1^{\alpha_1}\cdots x_d^{\alpha_d}}F^e\,,
\]
where $F$ is the $p$-th power map, $\alpha_k\le p^e-1$ for each $k$, and $\sum\alpha_k\equiv0\mod n$.
\end{proposition}

\begin{remark}
We use $F$ for the Frobenius endomorphism of the ring $M$. The condition $\sum\alpha_k\equiv0\mod n$, or equivalently $x_1^{\alpha_1}\cdots x_d^{\alpha_d}\in M$, implies that
\[
\frac{1}{x_1^{\alpha_1}\cdots x_d^{\alpha_d}}F^e\ \in\ \calF^e(M)\,.
\]
When $\alpha_k\le p^e-1$ for each $k$, the map displayed above stabilizes $N$ and thus induces an element of $\calF^e(M/N)$; we reuse $F$ for the $p$-th power map on $M/N$.
\end{remark}

\begin{proof}[Proof of Proposition~\ref{prop:veronese}]
In view of the above remark, it remains to establish that the given elements are indeed generators for $\calF^e(E)$. The canonical module of $R$ is
\[
\omega_R\ =\ \big(x_1\cdots x_dA\big)_{(n)}R
\]
and, indeed, $H^d_\frakm(\omega_R)=E$. Thus, Theorem~\ref{theorem:main} implies that
\[
\calF^e(E)\ =\ \omega_R^{(1-q)}F^e\,,
\]
where $q=p^e$. But $\omega_R^{(1-q)}$ is the completion of the $A_{(n)}$-module
\[
\left[\frac{1}{x_1^{q-1}\cdots x_d^{q-1}}A\right]_{(n)}=
\left(\frac{1}{x_1^{\alpha_1}\cdots x_d^{\alpha_d}}\ \mid\ \alpha_k\le q-1\text{ for each $k$, }\sum\alpha_k\equiv0\mod n\right)A_{(n)}\,,
\]
which completes the proof.
\end{proof}

\begin{example}
Consider $d=2$ and $n=3$ in Proposition~\ref{prop:veronese}, i.e.,
\[
R\ =\ \FF[[x^3,\ x^2y,\ xy^2,\ y^3]].
\]
Then $\omega=(x^2y,xy^2)R$ has order $3$ in the divisor class group of $R$; indeed,
\[
\omega^{(2)}=(x^4y^2,\ x^3y^3,\ x^2y^4)R\qquad\text{ and }\qquad\omega^{(3)}=(x^3y^3)R\,.
\]

\begin{enumerate}[\ \rm(1)]
\item If $p\equiv1\mod3$, then $\omega^{(1-q)}=(xy)^{1-q}R$ is cyclic for each $q=p^e$, and
\[
\calF^e(E)\ =\ \frac{1}{(xy)^{q-1}}F^e\,.
\]
Since
\[
\frac{1}{(xy)^{p-1}}F\ \circ\ \frac{1}{(xy)^{q-1}}F^e\ =\ \frac{1}{(xy)^{pq-1}}F^{e+1}\,,
\]
it follows that
\[
\calF(E)\ =\ R\left\{\frac{1}{(xy)^{p-1}}F\right\}\,.
\]

\item If $p\equiv2\mod3$ and $q=p^e$, then $\omega^{(1-q)}=(xy)^{1-q}R$ for $e$ even, and
\[
\omega^{(1-q)}\ =\ \left(\frac{1}{x^{q-3}y^{q-1}},\ \frac{1}{x^{q-2}y^{q-2}},\ \frac{1}{x^{q-1}y^{q-3}}\right)R
\]
for $e$ odd. The proof of Proposition~\ref{prop:q:gor} shows that $\calF(E)$ is generated by its elements of degree $\le2$, and hence
\[
\calF(E)\ =\ R\left\{\frac{1}{x^{p-3}y^{p-1}}F,\ \frac{1}{x^{p-2}y^{p-2}}F,\ \frac{1}{x^{p-1}y^{p-3}}F,\ \frac{1}{x^{p^2-1}y^{p^2-1}}F^2\right\}.
\]
In the case $p=2$, the above reads
\[
\calF(E)\ =\ R\left\{\frac{x}{y}F,\ F,\ \frac{y}{x}F,\ \frac{1}{x^3y^3}F^2\right\}.
\]

\item When $p=3$, one has
\[
\omega^{(1-q)}\ =\ \frac{1}{x^qy^q}\big(x^2y,\ xy^2\big)R\ =\ \left(\frac{1}{x^{q-2}y^{q-1}},\ \frac{1}{x^{q-1}y^{q-2}}\right)R
\]
for each $q=p^e$. In this case,
\[
\calF(E)\ =\ R\left\{\frac{1}{xy^2}F,\ \frac{1}{x^2y}F,\ \frac{1}{x^7y^8}F^2,\ \frac{1}{x^8y^7}F^2,\ \frac{1}{x^{25}y^{26}}F^3,\ \frac{1}{x^{26}y^{25}}F^3,\ \dots\right\},
\]
and $\calF(E)$ is not a finitely generated extension ring of $\calF^0(E)=R$; indeed,
\begin{align*}
\omega^{(1-q)}\cm\omega^{(1-q')}
&\ =\ \frac{1}{x^qy^q}\big(x^2y,\ xy^2\big)R\cm\frac{1}{x^{q'}y^{q'}}\big(x^2y,\ xy^2\big)R\\
&\ =\ \frac{1}{x^{qq'+q}y^{qq'+q}}\big(x^2y,\ xy^2\big)\cdot\big(x^{2q}y^q,\ x^qy^{2q}\big)R\\
&\ =\ \frac{1}{x^{qq'}y^{qq'}}\big(x^{q+2}y,\ x^{q+1}y^2,\ x^2y^{q+1},\ xy^{q+2}\big)R\\
&\ =\ \frac{1}{x^{qq'}y^{qq'}} (x^2y,\ xy^2)\cdot (x^q,\ y^q)R\\
&\ =\ (x^q,\ y^q)\ \omega^{(1-qq')}
\end{align*}
for $q=p^e$ and $q'=p^{e'}$, where $e$ and $e'$ are positive integers.
\end{enumerate}
\end{example}

\section{A determinantal ring}
\label{sec:determinantal}

Let $R$ be the determinantal ring $\FF[X]/I$, where $X$ is a $2\times 3$ matrix of variables over a field of characteristic $p>0$, and $I$ is the ideal generated by the size $2$ minors of~$X$. Set $\frakm$ to be the homogeneous maximal ideal of $R$. We show that the algebra of Frobenius operators $\calF(E)$ is not finitely generated over $\calF^0(E)=\hat{R}$; this proves Conjecture~3.1 of \cite{Katzman}. We also extend Fedder's calculation of the ideals $I^{[p]}:I$ to the ideals $I^{[q]}:I$ for all $q=p^e$.

The ring $R$ is isomorphic to the affine semigroup ring
\[
\FF\left[\begin{matrix}sx,&sy,&sz,\\tx,&ty,&tz\end{matrix}\right]
\ \subseteq\ \FF[s,t,x,y,z]\,.
\]
Using this identification, $R$ is the Segre product $A\#B$ of the polynomial rings $A=\FF[s,t]$ and $B=\FF[x,y,z]$. By \cite[Theorem~4.3.1]{GW}, the canonical module of~$R$ is the Segre product of the graded canonical modules $stA$ and $xyzB$ of the respective polynomial rings, i.e.,
\[
\omega_R\ =\ stA\ \#\ xyzB\ =\ (s^2txyz,\ st^2xyz)R\,.
\]
Let $e$ be a nonnegative integer, and $q=p^e$. Then
\[
\omega_R^{(1-q)}\ =\ \frac{1}{(st)^{q-1}}A\ \#\ \frac{1}{(xyz)^{q-1}}B
\]
is the $R$ module spanned by the elements
\[
\frac{1}{(st)^{q-1}x^ky^lz^m}
\]
with $k+l+m=2q-2$ and $k,l,m\le q-1$.

View $E$ as $M/N$ where $M=R_{s^2txyz}$, and $N$ is the $R$-submodule spanned by the elements $s^it^jx^ky^lz^m$ in $M$ that have at least one positive exponent. Then $\calF^e(E)$ is the left $\hat{R}$-module generated by
\[
\frac{1}{(st)^{q-1}x^ky^lz^m}F^e\,,
\]
where $F$ is the $p$-th power map, $k+l+m=2q-2$, and $k,l,m\le q-1$. Using this description, it is an elementary---though somewhat tedious---verification that $\calF(E)$ is not finitely generated over $\calF^0(E)$; alternatively, note that the symbolic powers of the height one prime ideals $(sx,sy,sz)\hat{R}$ and $(sx,tx)\hat{R}$ agree with the ordinary powers by \cite[Corollary~7.10]{BV}. Thus, the anticanonical cover of $\hat{R}$ is the ring $\calR$ with
\[
\calR_n\ =\ \frac{1}{(s^2txyz)^n}(sx,sy,sz)^n\hat{R}\,,
\]
and so
\[
\calT_e\ =\ \frac{1}{(s^2txyz)^{q-1}}(sx,sy,sz)^{q-1}\hat{R}\,.
\]
Thus,
\begin{align*}
\calT_{e_1}\cm\calT_{e_2}&\ =\ \frac{1}{(s^2txyz)^{q_1-1}}(sx,sy,sz)^{q_1-1}\cm
\frac{1}{(s^2txyz)^{q_2-1}}(sx,sy,sz)^{q_2-1}\\
&\ =\ \frac{1}{(s^2txyz)^{q_1q_2-1}}(sx,sy,sz)^{q_1-1}\cdot\left((sx,sy,sz)^{q_2-1}\right)^{[q_1]}\\
&\ =\ \frac{1}{(s^2txyz)^{q_1q_2-1}}(sx,sy,sz)^{q_1-1}\cdot\Big((sx)^{q_1}, (sy)^{q_1}, (sz)^{q_1}\Big)^{q_2-1}
\end{align*}
where $q_i=p^{e_i}$. We claim that
\[
\calT_e\ \neq\ \sum_{e_1=1}^{e-1}\calT_{e_1}\cm\calT_{e-e_1}\,.
\]
For this, it suffices to show that
\[
\frac{1}{(s^2txyz)^{q-1}}sx(sy)^{q/p-1}(sz)^{q-q/p-1}
\]
does not belong to $\calT_{e_1}\cm\calT_{e_2}$ for integers $e_i<e$ with $e_1+e_2=e$. By the description of $\calT_{e_1}\cm\calT_{e_2}$ above, this is tantamount to proving that
\[
sx(sy)^{q/p-1}(sz)^{q-q/p-1}\ \notin\ (sx,sy,sz)^{q_1-1}\cdot\Big((sx)^{q_1}, (sy)^{q_1}, (sz)^{q_1}\Big)^{q_2-1}\,,
\]
but this is essentially Example~\ref{example:cm}.\ref{example:dim3}.

\subsection*{Fedder's computation}

Let $A$ be the power series ring $\FF[[u,v,w,x,y,z]]$ for $\FF$ a field of characteristic $p>0$, and let $I$ be the ideal generated by the size $2$ minors of the matrix
\[
\begin{pmatrix}
u&v&w\\
x&y&z
\end{pmatrix}\,,
\]
In \cite[Proposition~4.7]{Fedder}, Fedder shows that
\[
I^{[p]}:I\ =\ I^{2p-2}+I^{[p]}\,.
\]
We extend this next by calculating the ideals $I^{[q]}:I$ for each prime power $q=p^e$.

\begin{proposition}
\label{prop:fedder}
Let $A$ be the power series ring $\FF[[u,v,w,x,y,z]]$ where $K$ a field of characteristic $p>0$. Let $I$ be the ideal of $A$ generated by $\Delta_1=vz-wy$, $\Delta_2=wx-uz$, and $\Delta_3=uy-vx$.
\begin{enumerate}[\rm(1)]
\item For $q=p^e$ and nonnegative integers $s,t$ with $s+t\le q-1$, one has
\[
y^sz^t(\Delta_2\Delta_3)^{q-1}\ \in\ I^{[q]}+x^{s+t}A\,.
\]

\item For $q,s,t$ as above, let $f_{s,t}$ be an element of $A$ with
\[
y^sz^t(\Delta_2\Delta_3)^{q-1}\ \equiv\ x^{s+t}f_{s,t}\mod I^{[q]}\,.
\]
Then $f_{s,t}$ is well-defined modulo $I^{[q]}$. Moreover, $f_{s,t}\in I^{[q]}:_AI$, and
\[
I^{[q]}:_AI\ =\ I^{[q]}+\big(f_{s,t}\mid s+t\le q-1\big)A\,.
\]
\end{enumerate}
\end{proposition}

For $q=p$, the above recovers Fedder's computation that $I^{[p]}:I\ =\ I^{2p-2}+I^{[p]}$, though for $q>p$, the ideal $I^{[p]}:I$ is strictly bigger than $I^{2p-2}+I^{[p]}$.

\begin{proof} (1) Note that the element
\[
y^sz^t(\Delta_2\Delta_3)^{q-1}\ =\ y^sz^t(wx-uz)^{q-1}(uy-vx)^{q-1}
\]
belongs to the ideals
\[
(x,u)^{2q-2}\ \subseteq\ (x^{q-1},u^q)\ \subseteq\ (x^{s+t},u^q)\,,
\]
and also to
\[
y^sz^t(x,z)^{q-1}(x,y)^{q-1}\ \subseteq\ y^sz^t(x^t,z^{q-t})(x^s,y^{q-s})
\ \subseteq\ (x^{s+t},z^q,y^q)\,.
\]
Hence,
\begin{align*}
y^sz^t(\Delta_2\Delta_3)^{q-1}\
&\ \in\ \big(x^{s+t},\ u^q\big)A\ \cap\ \big(x^{s+t},\ z^q,\ y^q\big)A\\
&\ =\ \big(x^{s+t},\ u^qz^q,\ u^qy^q\big)A\\
&\ \subseteq\ \big(x^{s+t},\ \Delta_1^q,\ \Delta_2^q,\ \Delta_3^q\big)A\,.
\end{align*}

(2) The ideals $I$ and $I^{[q]}$ have the same associated primes, \cite[Corollary~21.11]{24hours}. As $I$ is prime, it is the only prime associated to $I^{[q]}$. Hence $x^{s+t}$ is a nonzerodivisor modulo $I^{[q]}$, and it follows that $f_{s,t}\mod I^{[q]}$ is well-defined.

We next claim that
\[
I^{2q-1}\subseteq I^{[q]}\,.
\]
By the earlier observation on associated primes, it suffices to verify this in the local ring~$R_I$. But $R_I$ is a regular local ring of dimension $2$, so $IR_I$ is generated by two elements, and the claim follows from the pigeonhole principle. The claim implies that
\[
x^{s+t}f_{s,t}\, I\ \in\ I^{[q]}\,,
\]
and using, again, that $x^{s+t}$ is a nonzerodivisor modulo $I^{[q]}$, we see that $f_{s,t}I\subseteq I^{[q]}$, in other words, that $f_{s,t}\in I^{[q]}:_AI$ as desired.

By Theorem~\ref{theorem:main} and Remark~\ref{remark:fedder:ring}, one has the $R$-module isomorphisms
\[
\omega_R^{(1-q)}\ \cong\ \calF^e(E)\ \cong\ \frac{I^{[q]}:_AI}{I^{[q]}}\,.
\]
Choosing $\omega_R^{(-1)}=(x,y,z)R$, we claim that the map
\begin{align*}
(x,y,z)^{q-1}R\ &\to\ \frac{I^{[q]}:_AI}{I^{[q]}}\\
x^{q-1-s-t}y^sz^t\ &\mapsto\ f_{s,t}
\end{align*}
is an isomorphism. Since the modules in question are reflexive $R$-modules of rank one, it suffices to verify that the map is an isomorphism in codimension $1$. Upon inverting~$x$, the above map induces
\begin{align*}
R_x\ &\to\ \frac{I^{[q]}A_x:_{A_x}IA_x}{I^{[q]}A_x}\\
x^{q-1}\ &\mapsto\ (\Delta_2\Delta_3)^{q-1}
\end{align*}
which is readily seen to be an isomorphism since $IA_x=(\Delta_2,\Delta_3)A_x$.
\end{proof}

\section{Cartier algebras and gauge boundedness}
\label{sec:gauge}

For a ring $R$ of prime characteristic $p>0$, one can interpret $\calF^e(E)$ in a dual way as a collection of $p^{-e}$-linear operators on $R$. This point of view was studied by Blickle~\cite{Blickle:JAG} and Schwede~\cite{Schwede}.

\begin{definition}
Let $R$ be a ring of prime characteristic $p>0$. For each $e\ge0$, set $\calC^R_e$ to be set of additive maps $\phi\colon R\to R$ satisfying
\[
\phi(r^{p^e}x)\ =\ r\phi(x)\qquad\text{ for }r,x\in R\,.
\]
The \emph{total Cartier algebra} is the direct sum
\[
\calC^R\ =\ \bigoplus_{e\ge0}\calC^R_e\,.
\]
\end{definition}

For $\phi\in\calC^R_{e}$ and $\phi'\in\calC^R_{e'}$, the compositions $\phi\circ\phi'$ and $\phi'\circ\phi$ are elements of $\calC^R_{e+e'}$. This gives $\calC^R$ the structure of an $\NN$-graded ring; it is typically not a commutative ring. As pointed out in~\cite[2.2.1]{ABZ}, if $(R,\frakm)$ is an $F$-finite complete local ring, then the ring of Frobenius operators $\calF(E)$ is isomorphic to $\calC^R$.

Each $\calC^R_e$ has a left and a right $R$-module structure: for $\phi\in\calC^R_e$ and $r\in R$, we define $r\cdot\phi$ to be the map $x\mapsto r\phi(x)$, and $\phi\cdot r$ to be the map $x\mapsto\phi(rx)$.

\begin{definition}
Blickle~\cite{Blickle:JAG} introduced a notion of boundedness for Cartier algebras: Let~$R=A/I$ for a polynomial ring $A=\FF[x_1,\dots,x_d]$ over an $F$-finite field $\FF$. Set $R_n$ to be the finite dimensional $\FF$-vector subspace of $R$ spanned by the images of the monomials
\[
x_1^{\lambda_1}\cdots x_d^{\lambda_d}\qquad\text{ for }\ 0\le \lambda_j\le n\,.
\]
Following \cite{Anderson} and \cite{Blickle:JAG}, we define a map $\delta\colon R\to\ZZ$ by $\delta(r)=n$ if $r\in R_n\smallsetminus R_{n-1}$; the map~$\delta$ is a \emph{gauge}. If $I=0$, then $\delta(r)\le\deg(r)$ for each $r\in R$. We recall some properties from \cite[Proposition~1]{Anderson} and \cite[Lemma~4.2]{Blickle:JAG}:
\begin{align*}
\delta(r+r')\ &\le \ \max\{\delta(r),\ \delta(r')\}\,,\\
\delta(r\cdot r')\ &\le \ \delta(r)+\delta(r')\,.
\end{align*}

The ring $\calC^R$ is \emph{gauge bounded} if there exists a constant $K$, and elements $\phi_{e,i}$ in $\calC^R_e$ for each $e \ge 1$ generating $\calC^R_e$ as a left $R$-module, such that
\[
\delta(\phi_{e,i}(x))\ \le\ \frac{\delta(x)}{p^e}+K\qquad\text{ for each }e\text{ and }i\,.
\]
\end{definition}

\begin{remark}
We record two key facts that will be used in our proof of Theorem~\ref{theorem:gauge}:

\begin{enumerate}[\rm(1)]
\item If there exists a constant $C$ such that $I^{[p^e]}:_AI$ is generated by elements of degree at most $Cp^e$ for each $e\ge1$, then $\calC^R$ is gauge bounded; this is \cite[Lemma~2.2]{KatzmanZhang}.

\item If $\calC^R$ is gauge bounded, then for each ideal $\fraka$ of $R$, the $F$-jumping numbers of $\tau(R,\fraka^t)$ are a subset of the real numbers with no limit points; in particular, they form a discrete set. This is \cite[Theorem~4.18]{Blickle:JAG}.
\end{enumerate}
\end{remark}

We now prove the main result of the section:

\begin{theorem}
\label{theorem:gauge}
Let $R$ be a normal $\NN$-graded that is finitely generated over an $F$-finite field~$R_0$. (The ring $R$ need not be standard graded.)

Suppose that the anticanonical cover of $R$ is finitely generated as an $R$-algebra. Then~$\calC^R$ is gauge bounded. Hence, for each ideal $\fraka$ of $R$, the set of $F$-jumping numbers of $\tau(R,\fraka^t)$ is a subset of the real numbers with no limit points.
\end{theorem}

\begin{proof}
Let $A$ be a polynomial ring, with a possibly non-standard $\NN$-grading, such that $R=A/I$. It suffices to obtain a constant $C$ such that the ideals $I^{[p^e]}:_AI$ are generated by elements of degree at most $Cp^e$ for each $e\ge 1$.

There exists a ring isomorphism $\bigoplus_{e\ge0}\omega^{(1-p^e)}\cong\bigoplus_{e\ge0}(I^{[p^e]}:_AI)/I^{[p^e]}$ by Remark~\ref{remark:fedder:ring} that respects the $e$-th graded components. After replacing $\omega$ by an isomorphic $R$-module with a possible graded shift, we may assume that the isomorphism above induces degree preserving $R$-module isomorphisms $\omega^{(1-p^e)}\cong(I^{[p^e]}:_AI)/I^{[p^e]}$ for each $e\ge0$. While $\omega$ is no longer canonically graded, we still have the finite generation of the anticanonical cover $\bigoplus_{n\ge0}\omega^{(-n)}$. It suffices to check that there exists a constant $C$ such that $\omega^{(1-p^e)}$ is generated, as an $R$-module, by elements of degree at most $Cp^e$.

Choose finitely many homogeneous $R$-algebra generators $z_1,\dots,z_k$ for $\bigoplus_{n\ge0}\omega^{(-n)}$, say with $z_i\in\omega^{(-j_i)}$. Set $C$ to be the maximum of $\deg z_1,\dots,\deg z_k$. Then the monomials
\[
\bsz^\bslambda\ =\ z_1^{\lambda_1}z_2^{\lambda_2}\cdots z_k^{\lambda_k} \qquad\text{ with }\sum\lambda_ij_i=p^e-1
\]
generate the $R$-module $\omega^{(1-p^e)}$, and it is readily seen that
\[
\deg\bsz^\bslambda\ =\ \sum\lambda_i\deg z_i\ \le\ C\sum\lambda_i\ \le\ C(p^e-1)\,.
\]
By \cite[Lemma~2.2]{KatzmanZhang}, it follows that $\calC^R$ is gauge bounded; the assertion now follows from \cite[Theorem~4.18]{Blickle:JAG}.
\end{proof}

\begin{corollary}
Let $R$ be the determinantal ring $\FF[X]/I$, where $X$ is a matrix of indeterminates over an $F$-finite field $\FF$ of prime characteristic, and $I$ is the ideal generated by the minors of~$X$ of an arbitrary but fixed size. Then, for each ideal $\fraka$ of $R$, the set of $F$-jumping numbers of $\tau(R,\fraka^t)$ is a subset of the real numbers with no limit points.
\end{corollary}

\begin{proof}
Since $R$ is a determinantal ring, the symbolic powers of the ideal $\omega^{(-1)}$ agree with the ordinary powers by \cite[Corollary~7.10]{BV}. Hence the anticanonical cover of $R$ is finitely generated, and the result follows from Theorem~\ref{theorem:gauge}.
\end{proof}

\begin{remark}
It would be natural to remove the hypothesis that $R$ is graded in Theorem~\ref{theorem:gauge}. However, we do not know how to do this: when $R$ is not graded, it is unclear if one can choose gauges that are compatible with the ring isomorphism
\[
\bigoplus_{e\ge0}\omega^{(1-p^e)}\ \cong\ \bigoplus_{e\ge0}(I^{[p^e]}:_AI)/I^{[p^e]}\,.
\]
\end{remark}

\section{Linear growth of Castelnuovo-Mumford regularity for rings of finite Frobenius representation type}
\label{sec:regularity}

Let $A$ be a standard graded polynomial ring over a field $\FF$, with homogeneous maximal ideal $\frakm$. We recall the definition of the Castelnuovo-Mumford regularity of a graded module following~\cite[Chapter~4]{Eisenbud}:

\begin{definition}
Let $M=\bigoplus_{d\in\QQ} M_d$ be a graded $A$-module. If $M$ is Artinian, we set
\[
\reg M\ =\ \max\{d\mid M_d\neq 0\}\,;
\]
for an arbitrary graded module we define
\[
\reg M\ =\ \max_{k\ge 0}\{\reg H^k_\frakm(M)+k\}\,.
\]
\end{definition}

\begin{definition}
Let $I$ and $J$ be homogeneous ideals of $A$. We say that the regularity of~$A/(I+J^{[p^e]})$ has \emph{linear growth} with respect to $p^e$, if there is a constant $C$, such that
\[
\reg A/(I+J^{[p^e]})\ \le\ Cp^e\qquad\text{ for each }e\ge 0\,.
\]
It follows from \cite[Corollary~2.4]{KatzmanZhang} that if $\reg A/(I+J^{[p^e]})$ has linear growth for each homogeneous ideal $J$, then $\calC^{A/I}$ is gauge-bounded.
\end{definition}

\begin{remark}
Let $R=A/I$ for a homogeneous ideal $I$. We define a grading on the bimodule~$R^{(e)}$ introduced in Remark~\ref{remark:adjoint}: when an element $r$ of $R$ is viewed as an element of
$R^{(e)}$, we denote it by $r^{(e)}$. For a homogeneous element $r\in R$, we set
\[
\deg'r^{(e)}\ =\ \frac{1}{p^e}\deg r\,.
\]
For each ideal $J$ of $R$, one has an isomorphism
\[
\CD
R^{(e)}\otimes_R R/J@>\cong>>R/J^{[p^e]}
\endCD
\]
under which ${r^{(e)}}\otimes\bar{s}\mapsto\bar{rs^{p^e}}$. To make this isomorphism degree-preserving for a homogeneous ideal $J$, we define a grading on $R/J^{[p^e]}$ as follows:
\[
\deg'\bar{r}\ =\ \frac{1}{p^e}\deg\bar{r}\qquad\text{ for a homogeneous element }r\text{ of }R\,.
\]
\end{remark}

The notion of finite Frobenius representation type was introduced by Smith and Van den Bergh~\cite{SV}; we recall the definition in the graded context:

\begin{definition}
Let $R$ be an $\NN$-graded Noetherian ring of prime characteristic $p$. Then $R$ has \emph{finite graded Frobenius-representation type} by finitely generated $\QQ$-graded $R$-modules $M_1,\dots,M_s$, if for every $e\in\NN$, the $\QQ$-graded $R$-module $R^{(e)}$ is isomorphic to a finite direct sum of the modules $M_i$ with possible graded shifts, i.e., if there exist rational numbers $\alpha^{(e)}_{ij}$, such that there exists a~$\QQ$-graded isomorphism
\[
R^{(e)}\ \cong\ \bigoplus_{i,j}M_i\big(\alpha^{(e)}_{ij}\big)\,.
\]
\end{definition}

\begin{remark}
\label{remark:bounds:shifts}
Suppose $R$ has finite graded Frobenius-representation type. With the notation as above, there exists a constant $C$ such that
\[
\alpha^{(e)}_{ij}\ \le\ C\qquad\text{ for all }e,i,j\,;
\]
see the proof of~\cite[Theorem~2.9]{TT}.
\end{remark}

We now prove the main result of this section; compare with~\cite[Theorem~4.8]{TT}.

\begin{theorem}
Let $A$ be a standard graded polynomial ring over an $F$-finite field of characteristic $p>0$. Let $I$ be a homogeneous ideal of $A$.

Suppose $R=A/I$ has finite graded $F$-representation type. Then $\reg A/(I+J^{[p^e]})$ has linear growth for each homogeneous ideal $J$. In particular,~$\calC^R$ is gauge bounded, and for each ideal $\fraka$ of $R$, the set of $F$-jumping numbers of $\tau(R,\fraka^t)$ is a subset of the real numbers with no limit points.
\end{theorem}

\begin{proof}
We use $J$ for the ideal of $A$, and also for its image in $R$. Let $a'(H^k_\frakm(R/J^{[p^e]}))$ denote the largest degree of a nonzero element of $H^k_\frakm(R/J^{[p^e]})$ under the $\deg'$-grading, i.e.,
\[
a'(H^k_\frakm(R/J^{[p^e]}))\ =\ \frac{1}{p^e}\reg H^k_\frakm(R/J^{[p^e]})\,.
\]
Since we have degree-preserving isomorphisms $R^{(e)}\otimes_R R/J\cong R/J^{[p^e]}$, and
\[
R^{(e)}\ \cong\ \bigoplus_{i,j}M_i(\alpha^{(e)}_{ij})\,,
\]
it follows that
\begin{align*}
H^k_\frakm(R/J^{[p^e]})\ &\cong\ H^k_\frakm(R^{(e)}\otimes_R R/J)\\
&\cong\ \bigoplus_{i,j}H^k_\frakm\big(M_i(\alpha^{(e)}_{ij})\otimes_R R/J\big)\\
&\cong\ \bigoplus_{i,j}H^k_\frakm(M_i/JM_i)(\alpha^{(e)}_{ij})\,.
\end{align*}
The numbers $\alpha^{(e)}_{ij}$ are bounded by Remark~\ref{remark:bounds:shifts}; thus,
\[
a'(H^k_\frakm(R/J^{[p^e]}))\ \le\ \max_i\{a'(H^k_{\frakm}(M_i/JM_i))+C\}\,.
\]
Since there are only finitely many modules $M_i$ and finitely many homological indices $k$, it follows that $a'(H^k_\frakm(R/J^{[p^e]}))\le C'$, where $C'$ is a constant independent of $e$ and $k$. Hence
\[
\reg H^k_\frakm(R/J^{[p^e]})\ \le\ C'p^e\qquad\text{ for all }e,k\,,
\]
and so
\[
\reg A/(I+J^{[p^e]})\ =\ \max_k\{\reg H^k_\frakm(R/J^{[p^e]})+k\}\le C'p^e+\dim A\,.
\]
This proves that $\reg A/J^{[p^e]}$ has linear growth; \cite[Corollary~2.4]{KatzmanZhang} implies that $\calC^R$ is gauge bounded, and the discreetness assertion follows from \cite[Theorem~4.18]{Blickle:JAG}.
\end{proof}


\end{document}